\documentclass[11pt,reqno]{amsart}

\usepackage{amsmath,amssymb,amsthm,mathtools}
\usepackage{newtxtext,newtxmath}
\usepackage{parskip}

\usepackage[margin=1in]{geometry}
\usepackage[pagebackref, colorlinks=true,linkcolor=red,citecolor=blue,urlcolor=blue]{hyperref}

\newtheorem{theorem}{Theorem}[section]
\newtheorem{proposition}[theorem]{Proposition}
\newtheorem{lemma}[theorem]{Lemma}
\newtheorem{corollary}[theorem]{Corollary}

\newtheorem*{maintheoremA}{Theorem A}
\theoremstyle{definition}

\newtheorem{example}[theorem]{Example}
\newtheorem{remark}[theorem]{Remark}

\newtheorem{notation}[theorem]{Notation}
\newtheorem{notrem}[theorem]{Notation/Remark}

\DeclareMathOperator{\soc}{soc}
\newcommand{\NS}{\mathrm{NS}}

\begin{document}

\title[WLP for AG algebras of codimension three]
{On the weak Lefschetz property of Artinian Gorenstein
algebras of codimension three in arbitrary characteristic}

 \author[Omkar Javadekar]{Omkar Javadekar}
 	\address{Chennai Mathematical Institute, Siruseri, Tamilnadu 603103. INDIA}
 	\email{omkarjavadekar@gmail.com, omkarj@cmi.ac.in}



\begin{abstract}
Let $\mathsf k$ be a field, $S=\mathsf k[x,y,z]$, and $R=S/I$ be a standard graded Artinian Gorenstein $\mathsf k$-algebra of codimension three. The $h$-vector of such an algebra is known to be symmetric and unimodal. Mir\'o-Roig proved that if $\mathsf k$ is algebraically closed of characteristic zero and the $h$-vector of $R$ has at least three peaks, then $R$ has the weak Lefschetz property. We first observe that this statement holds over any infinite field of arbitrary characteristic, and that it follows from a theorem of Iarrobino--Kanev by an argument which is already used in special cases in the literature. We then give a second, characteristic-free proof, following the strategy of Migliore--Zanello, which is elementary and does not use the theorem of Iarrobino--Kanev, and hence does not use the Buchsbaum--Eisenbud structure theorem. 
\end{abstract}
\maketitle

\section{Introduction}
Let $R$ be an Artinian standard graded algebra over a field $\mathsf k$ with socle degree $d$.
We say that $R$ has the \emph{weak Lefschetz property (WLP)} if there exists a linear form $\ell \in R_1$ such that given any $i \geq 0$ the multiplication map $\times \ell: R_i \to R_{i+1}$ is injective or surjective.
Despite the very simple looking definition, determining whether a given algebra has the WLP turns out to be a hard problem. For an overview of the developments in this area, we refer the interested reader to the survey of Migliore--Nagel \cite{MigNagTour}.

One of the classes of algebras for which the WLP has been studied extensively is that of Artinian Gorenstein algebras. Over an infinite field, every Artinian quotient of a polynomial ring in two variables has the WLP in any characteristic. Roughly speaking, this is because, modulo a general linear form, we obtain a quotient of a polynomial ring in one variable, which forces the multiplication maps to have maximal rank. In three variables, modulo a general linear form, we are left with a quotient of a polynomial ring in two variables, where we again have good control. This is the idea behind our second proof. In codimension four and higher, in both characteristic zero and positive characteristic, there are Artinian Gorenstein algebras which do not satisfy the WLP (see, e.g., \cite[Remark 2.9, Example 3.10]{HMNW}). So, codimension three is the only case in which the question remains open, and it has been conjectured that in characteristic zero every Artinian Gorenstein algebra of codimension three has the WLP (see, e.g.,~\cite{BMMNZ, MigNagTour, MiglioreZanello}).

In characteristic zero, the question has been settled in several special cases, but it remains open in general. Harima--Migliore--Nagel--Watanabe \cite{HMNW} proved that every complete intersection of codimension three has the WLP, using rank two vector bundles on $\mathbb P^2$ and the Grauert--M\"ulich theorem. Boij--Migliore--Mir\'o-Roig--Nagel--Zanello \cite{BMMNZ} reduced the general problem to compressed Gorenstein algebras of odd socle degree and determined, in every characteristic, which algebras with Hilbert function $(1,3,6,6,3,1)$ have the WLP. Recently, Mir\'o-Roig \cite{MiroRoig} proved that if $\mathsf k$ is algebraically closed of characteristic zero and the $h$-vector of $R$ has at least three peaks, then $R$ has the WLP.

In general, the weak Lefschetz property is characteristic sensitive. Even for monomial complete intersections, whether the WLP holds depends on the characteristic (see \cite{BrennerKaid, CookII, KustinVraciu}), and the same happens outside the class of complete intersections (see, e.g., \cite[Theorem 3.8]{BMMNZ}).

In this article we consider Mir\'o-Roig's theorem over an arbitrary infinite field. Our main theorem is as follows.

\begin{maintheoremA}[see Theorem~\ref{thm:main}]
Let $\mathsf k$ be an infinite field, $S=\mathsf k[x,y,z]$, and $R=S/I$ be a standard graded Artinian Gorenstein $\mathsf k$-algebra of codimension three. If the $h$-vector of $R$ attains its maximum at least three times, then $R$ has the weak Lefschetz property.
\end{maintheoremA}
We give two proofs of this theorem. The first proof is a short argument combining a theorem of Iarrobino--Kanev with the criterion of Migliore--Mir\'o-Roig--Nagel recalled in Remark~\ref{rem:middle}, and it is the shortest route to the theorem that we know. We use the notation of Section~\ref{sec:prelim}:   $\ell$ is a general linear form, $d$ is the socle degree of $R$, the $h$-vector of $R$ attains its maximum in the degrees $t, t+1, \ldots, s$, and $j_0 = \lfloor (d-1)/2 \rfloor$.
\begin{proof}[{\rm \textbf{First proof of Theorem~{A}}}]
We may assume that $\mathsf k$ is algebraically closed. Indeed, let $\overline{\mathsf k}$ denote an algebraic closure of $\mathsf k$. The $h$-vector of $R$ does not change on passing from $R$ to $R \otimes_{\mathsf k} \overline{\mathsf k}$. Fix a basis of $R_i$ for every $i$, and let $\ell = ax+by+cz$. In these bases, the entries of the matrix of $\times \ell \colon R_{i-1} \to R_i$ are linear forms in $a,b,c$ with coefficients in $\mathsf k$. Therefore, the maximal minors of this matrix are polynomials in $a,b,c$ with coefficients in $\mathsf k$. Moreover, $\times \ell$ has maximal rank in degree $i$ if and only if one of these minors does not vanish at $(a,b,c)$. Also, note that only finitely many degrees $i$ occur, since $R$ is Artinian. Now, suppose that $R \otimes_{\mathsf k} \overline{\mathsf k}$ has the weak Lefschetz property. Then there is a point of $\overline{\mathsf k}^3$ at which, for every $i$, one of these minors does not vanish. Let $P$ be the product of these minors, one for each $i$. Then $P$ is a nonzero polynomial with coefficients in $\mathsf k$. Since $\mathsf k$ is infinite, $P$ does not vanish at all the points of $\mathsf k^3$. Hence, there is a linear form $\ell$ with coefficients in $\mathsf k$ at which $\times \ell$ has maximal rank in every degree, and $R$ has the weak Lefschetz property.

Now, by hypothesis, we have $s \geq t+2$. By Lemma~\ref{lem:peaks}(a), we have $t+s=d$, and hence $d \geq 2t+2$. Applying \cite[Theorem 5.31]{IarrobinoKanev}, we get a saturated ideal $\widetilde I \subseteq I$ of height two with $\widetilde I_i = I_i$ for $t \leq i \leq s$. Since $\widetilde I$ is saturated, $\mathfrak m$ is not an associated prime of $S/\widetilde I$. Since $\mathsf k$ is infinite, the general linear form $\ell$ therefore avoids all the  associated primes of $S/\widetilde I$, and is a nonzerodivisor on $S/\widetilde I$. Thus, the multiplication by $\ell$ on $S/\widetilde I$ is injective in every degree. By Lemma~\ref{lem:peaks}(b), we have $t \leq j_0$ and $j_0+1 \leq s$. Therefore, $R_{j_0} = (S/\widetilde I)_{j_0}$ and $R_{j_0+1} = (S/\widetilde I)_{j_0+1}$. It follows that $\times \ell \colon R_{j_0} \to R_{j_0+1}$ is injective. Now Remark~\ref{rem:middle} shows that $R$ has the weak Lefschetz property.
\end{proof}
There is nothing essentially new in the above argument. It is used for particular $h$-vectors in \cite[proof of Theorem 3.6]{MiglioreZanello}, via \cite[Theorem 5.77]{IarrobinoKanev}, and in the proof of \cite[Lemma 2.15]{AAIY}. The results of \cite{IarrobinoKanev} are proved over an algebraically closed field of arbitrary characteristic. What we observe is that the argument therefore gives Theorem A in the generality stated.

The rest of this article gives a second proof of Theorem A, which does not use \cite[Theorem 5.31]{IarrobinoKanev}, and hence does not use the structure theorem of Buchsbaum--Eisenbud \cite{BuchsbaumEisenbud} on which the latter rests. Our proof is elementary, and is organized around a single invariant. The idea of the proof is as follows. Fix a general linear form $\ell$ and let $c_i$ denote the dimension of the cokernel of $\times \ell \colon R_{i-1} \to R_i$. We show that if the $h$-vector of $R$ has at least two peaks, say in degrees $t, t+1, \ldots, s$, the $c_i$ take a common value $\gamma$ for $t+1 \leq i \leq s$ (Proposition~\ref{prop:constant}), and that $R$ has the weak Lefschetz property precisely when $\gamma = 0$ (Corollary~\ref{cor:wlp-gamma}). When there are at least three peaks, this constancy is \cite[Lemma 2.15]{AAIY}, proved there using \cite[Theorem 5.31]{IarrobinoKanev}. See Remark~\ref{rem:twopeaks} for the case of exactly two peaks.

Let $J$ denote the image of $I$ in $S/\langle \ell\rangle \cong \mathsf k[u,v]$. Then $c_i = (i+1)-\dim_{\mathsf k}(J_i)$. When there are at least three peaks and $\gamma \geq 1$, the dimensions $\dim_{\mathsf k}(J_i)$ grow as slowly as they can over the peak degrees, and this forces the elements of $I_s$ to have a common factor of positive degree, which is ruled out by a result of Migliore--Zanello (Remark~\ref{rem:MZgcd}). This strategy is that of Migliore--Zanello \cite[Section 3]{MiglioreZanello}. There the common factor is produced from Davis's theorem, while we produce it from Lemma~\ref{lem:growth}.

We use three results from the literature, recalled in Section~\ref{sec:prelim}. The first is Stanley's theorem that $h$-vectors of codimension three Gorenstein algebras are unimodal \cite{StanleyHF}. We use it in the form of the elementary proof given by Zanello \cite{Zanello}, which is valid over any field, and which does not use the Buchsbaum--Eisenbud structure theorem. The second is the criterion \cite[Proposition 2.1]{MiglioreMiroRoigNagel} mentioned above. The third is \cite[Lemma 2.12]{MiglioreZanello}. All three hold in arbitrary characteristic, and this is what makes the second proof characteristic free. 

We remark that the hypothesis on the number of peaks cannot be weakened in positive characteristic. In Example~\ref{ex:char2}, we recall the complete intersection $R=\mathsf k[x,y,z]/\langle x^2,y^2,z^2\rangle$ in characteristic two, whose $h$-vector $(1,3,3,1)$ has two peaks and which does not have the WLP. Thus, the theorem we prove is best possible if there is no restriction on characteristic.

The article is organized as follows. Section~\ref{sec:prelim} fixes the notation and recalls the results that we use. Section~\ref{sec:twovar} is about ideals in a polynomial ring in two variables. Section~\ref{sec:constant} introduces the invariant $\gamma$. Section~\ref{sec:divisor} contains the second proof of Theorem A.

\section*{Acknowledgements}
The author would like to thank Meghana Bhat for her comments on earlier drafts of this article, some useful discussions related to Gorenstein algebras, and for sharing resources on Lefschetz properties, which helped him learn the basic ideas and get familiar with some of the problems in this area. The author gratefully acknowledges support from a postdoctoral fellowship from Chennai Mathematical Institute and additional support from the Infosys Foundation.
\section{Preliminaries}\label{sec:prelim}

Throughout this article, unless stated otherwise, $\mathsf k$ denotes an infinite field. To avoid ambiguity, whenever the infiniteness of $\mathsf k$ is not needed, we make it explicit. Also, throughout $S = \mathsf k[x,y,z]$ denotes the standard graded polynomial ring in three variables over $\mathsf k$.  We let $\mathfrak{m}$ denote the unique homogeneous maximal ideal $\langle x, y, z\rangle$ of $S$. For any graded object $M$, we use $M_i$ to denote its $i^{th}$ graded component. All rings considered in this article are standard graded, and all ideals are graded.

Let $I \subseteq \mathfrak m$ be a homogeneous ideal of $S$ and let $R = S/I$ be Artinian, with $R_i=S_i/I_i$. The \emph{socle degree} of $R$ is the largest integer $d$ with $R_d \neq 0$, and the \emph{$h$-vector} of $R$ is $h(R)=(h_0,h_1,\dots,h_d)$, where $h_i = \dim_{\mathsf k}(R_i)$. The number $h_1$ is the \emph{codimension} of $R$, so that $R$ has codimension three if and only if $I \subseteq \mathfrak m^2$. We write $\soc(R)=0:_R\mathfrak m$, and we say that $R$ is \emph{Gorenstein} if $\dim_{\mathsf k}(\soc(R))=1$. In that case $\soc(R)=R_d$, and $h_d=h_0=1$.

From now on, unless stated otherwise, $R=S/I$ denotes a standard graded Artinian Gorenstein $\mathsf k$-algebra of codimension three, with socle degree $d$ and $h$-vector $h(R)=(h_0,\dots,h_d)$. We refer the reader to \cite{BrunsHerzog} for the standard facts we use. The first of these is Gorenstein duality.

\begin{remark}\label{rem:AG-facts}
For $0\leq i \leq d$ the bilinear pairing $$ R_i \times R_{d-i} \longrightarrow R_d \cong \mathsf k, \qquad (a,b) \longmapsto ab $$
is a perfect pairing of $\mathsf k$-vector spaces. In particular $h_i = h_{d-i}$ for all $i$, and hence $h(R)$ is symmetric.
\end{remark}

We next recall the three results from the literature that we use. All of them hold in arbitrary characteristic.

\begin{remark}\label{rem:Stanley}
Let $\mathsf k$ be any field (not necessarily infinite) and let $R$ be a standard graded Artinian Gorenstein algebra of codimension at most three. Then $h(R)$ is unimodal. In codimension two, $R$ is a complete intersection and $h(R)$ is unimodal, and in codimension at most one the statement is clear. For codimension three this is a theorem of Stanley \cite{StanleyHF}. We use it in the form of the elementary proof given by Zanello \cite{Zanello}, which is valid over any field, and which does not depend on the structure theorem of Buchsbaum--Eisenbud \cite{BuchsbaumEisenbud}. 
\end{remark}

\begin{remark}[see {\cite[Proposition 2.1]{MiglioreMiroRoigNagel}}]\label{rem:middle}
Set $j_0 = \lfloor (d-1)/2\rfloor$. Then $R$ has the WLP if and only if for a general linear form $\ell$ the map $R_{j_0} \xrightarrow{\times \ell} R_{j_0+1}$ is injective.
\end{remark}

\begin{remark}[{\cite[Lemma 2.12]{MiglioreZanello}}]\label{rem:MZgcd}
If $i>d/2$, then the elements of $I_i$ have no common factor of positive degree.
\end{remark}

The hypothesis $\operatorname{char}(\mathsf k)=0$ made in \cite{MiglioreZanello} is not used in the proof of this lemma. Beyond the exact sequence relating $R/I$, $R/(I:F)$ and $R/(I,F)$ for a homogeneous form $F$, the only input is the unimodality of $h$-vectors of Gorenstein algebras of codimension at most three (see \cite[Remark 2.13]{MiglioreZanello}), which holds over any field by Remark~\ref{rem:Stanley}.

The notation we set up next is the main idea of our approach. Instead of the ranks of the multiplication maps by linear forms, we focus on the dimensions of their cokernels. These cokernels are the graded components of a quotient of a polynomial ring in two variables, which reduces the problem to understanding homogeneous ideals in two variables.

\begin{notrem}\label{notrem:c_i}
Let $\ell$ be a nonzero linear form in $R$.
\begin{enumerate}
    \item[{\rm (a)}] We have the four-term exact sequence
$$0 \to (0 :_R \,\ell)(-1) \to R(-1) \xrightarrow[]{ \times \ell} R \to R/\langle \ell \rangle \to 0.$$

     \item[{\rm (b)}] We set
$$c_i = \dim_{\mathsf k} \left( \left(\dfrac{R}{\langle \ell \rangle}\right)_i \right) = \dim_{\mathsf k} \left( \left(\dfrac{S}{I+\langle \tilde\ell \rangle}\right)_i \right),$$
where $\tilde\ell$ is any lift of $\ell$ to $S$. For simplicity we may use $\ell$ itself to denote the lift. Note that the number $c_i$ depends on the choice of $\ell$, which will always be clear from the context, so we do not indicate it in the notation.

        \item[{\rm (c)}] By (a), $c_i=0$ if and only if $\times\ell \colon R_{i-1}\to R_i$ is surjective, and $c_i = h_i-h_{i-1}$ if and only if this map is injective.

    \item[{\rm (d)}] Instead of $x,y,z$ we may choose coordinates $u,v,\ell$ so that $S=\mathsf k[u,v,\ell]$, and then identify $S/\langle \ell\rangle$ with the polynomial ring $T=\mathsf k[u,v]$. Let $J$ be the image of $I$ under $S \to S/\langle \ell\rangle \cong T$. Then
$$c_i = \dim_{\mathsf k}(T_i)-\dim_{\mathsf k}(J_i) = (i+1)-\dim_{\mathsf k}(J_i).$$
Note that $J_i$ is the image of $I_i$ under $S_i \to S_i/\ell S_{i-1}\cong T_i$. Hence, $J_i=0$ if and only if $I_i \subseteq \ell S_{i-1}$.
\end{enumerate}
\end{notrem}

\begin{notation}\label{notation:tspeaks}
 We set
$$
  S_R = \max\{ h_i \mid 0 \leq i \leq d\}, \qquad
  t = \min\{i \mid h_i = S_R\}, \qquad
  s = \max\{i \mid h_i = S_R\} .
$$
\end{notation}

Since $h(R)$ is unimodal, $h_i = S_R$ for $t \leq i \leq s$, while $h_i<S_R$ for $i<t$ and for $i>s$. We call $\{t,t+1,\dots,s\}$ the \emph{peak degrees} of $h(R)$, the integer $S_R$ the \emph{Sperner number} of $R$, and $\NS_R = s-t+1$ the \emph{number of peaks} of $h(R)$. Throughout, $j_0 = \lfloor (d-1)/2\rfloor$ as in Remark~\ref{rem:middle}.

\begin{lemma}\label{lem:peaks}
With the notation as above, the following hold.
\begin{enumerate}
\item[{\rm (a)}] $t+s=d$.
\item[{\rm (b)}] If $\NS_R \geq 2$, then $t \leq j_0$ and $j_0+1 \leq s$. In particular $h_{j_0}=h_{j_0+1}=S_R$.
\item[{\rm (c)}] If $\NS_R \geq 2$, then $s>d/2$.
\end{enumerate}
\end{lemma}

\begin{proof}
(a) Since $h(R)$ is symmetric, $h_i = h_{d-i}$ for all $i$. As $h_i<h_s$ for every $i>s$, the definition of $t$ forces $t=d-s$.

(b) Note that when $d$ is odd, $j_0$ and $j_0+1$ are the two central coordinates of $h(R)$; and when $d$ is even, $j_0+1$ is the central coordinate of $h(R)$. As noted in Remark \ref{rem:Stanley}, $h(R)$ is unimodal. Therefore, by our hypothesis $\NS_R\geq 2$, and the symmetry and unimodality of $h(R)$, we have $h_{j_0} = h_{j_0+1} = S_R$ when $d$ is odd, and $h_{j_0}=h_{j_0+1} = h_{j_0+2}=S_R$ when $d$ is even. In either case, it follows that $t \leq j_0 < j_0+1 \leq s$ and $h_{j_0} = h_{j_0+1} = S_R$.

(c) By (a) we have $t=d-s$, and $\NS_R\geq 2$ gives $s \geq t+1 = d-s+1$. Hence $2s \geq d+1$, that is, $s>d/2$.
\end{proof}

\section{Ideals in two variables}\label{sec:twovar}

In this section, $\mathsf k$ is any field (not necessarily infinite) and $T=\mathsf k[u,v]$.
The first lemma is a statement about minimal growth together with its equality case. The conclusion we draw from it in Section~\ref{sec:divisor}, namely that minimal growth forces a common factor, is the conclusion drawn from Davis's theorem in \cite[Lemma 2.11]{MiglioreZanello}. The proof we give is direct and elementary.

\begin{lemma}\label{lem:growth}
Let $\mathsf k$ be any field (not necessarily infinite), $T = \mathsf k[u,v]$ and $\mathfrak n=\langle u,v\rangle$. Fix $i\geq 0$ and let $V \subseteq T_i$ be a $\mathsf k$-vector subspace of dimension $\delta \geq 1$. Then
$$
  \dim_{\mathsf k} (T_1\cdot V) \;\ge\; \delta + 1 .
$$
Furthermore, equality holds if and only if
\[
  V=g\cdot T_{\,i-e}, \qquad \text{where } g = \gcd(V)
  \text{ and } e=\deg (g)=i-\delta +1 .
\]
\end{lemma}

\begin{proof}
    Consider $T$ with the lexicographic order with $u>v$. Then a $\mathsf k$-basis of $V$ can be written in the form such that the leading terms of the basis elements are $u^{r_1}v^{i-r_1}, \ldots, u^{r_\delta}v^{i-r_\delta}$ with $r_1>r_2>\cdots >r_\delta$. Hence, $T_1\cdot V$ contains polynomials with leading terms $u^{r_1+1}v^{i-r_1}, \ldots, u^{r_\delta+1}v^{i-r_\delta}$ obtained by multiplying each of the basis element of $V$ by $u$, and also the polynomial with the leading term $u^{r_\delta}v^{i-r_\delta+1}$ obtained by multiplying the last basis vector by $v$. Since all these $\delta+1$ leading terms are distinct, it follows that $\dim_{\mathsf k}(T_1 \cdot V) \geq \delta+1$. 

    For the second part, first of all note that given any nonzero homogeneous $g \in  T$, the ideal $\langle g \rangle$ is a free $T$-module of rank one. Thus, $\dim_{\mathsf k}(\langle g\rangle_{\deg(g)+j})= j+1$ for all $j\geq 0$. This shows that if $V= g \cdot T_{i-e}$ as in the statement, then $\dim_{\mathsf k} (T_1\cdot V)= \dim_{\mathsf k} (g\cdot T_{i-e+1})=\dim_{\mathsf k} (\langle g \rangle_{\deg(g)+i-e+1})=(i-e+1)+1=\delta + 1 $. This proves one implication of the ``if and only if'' statement. 
    
    To prove the other implication, assume that $\dim_{\mathsf k} (T_1\cdot V) = \delta + 1$ and let $g=\gcd(V)$ with $\deg(g)=e$. We want to show that $V = g \cdot T_{\,i-e}$ and $e = \deg (g) = i-\delta + 1$.
    
Let $W=\{f/g \mid f\in V\}$. Then we have  $V=g\cdot W$. We claim that $W=T_{i-e}$. If the claim holds, then $\dim_{\mathsf k}(\langle g\rangle_{\deg(g)+(i-\deg(g)+1)})=\dim_{\mathsf k}(T_1 \cdot g\cdot W)=\dim_{\mathsf k}(T_1 \cdot V)= \delta+1$ forces $(i-\deg(g)+1)+1=\delta+1$, i.e., $\deg(g)=i-\delta+1$, as required. Thus, to complete the proof, it suffices to prove that $W=T_{i-e}$, which is what we do now. 

Note that $W$ is a $\mathsf k$-vector subspace of $T_{i-e}$ of dimension $\delta$ with $\gcd(W)=1$. Also, we have $\dim_{\mathsf k}(T_1 \cdot W)= \dim_{\mathsf k}(g \cdot(T_1 \cdot W))= \dim_{\mathsf k}(T_1 \cdot V)=\delta+1$. If $\deg(g)=i$, then we must have $V=\langle g\rangle_{\deg(g)}=\mathsf kg$, and hence $W=T_0$ and we are done. Therefore, we may assume that $\deg(g)<i$. Then $W \subseteq \mathfrak n$. Since $\gcd(W)=1$ and $W\neq 0$, we have that the ideal $\langle W \rangle$ of $T$ generated by the elements of $W$ is $\mathfrak n$-primary. If $\deg(g)=i-1$, then the $\mathfrak n$-primary condition on $\langle W\rangle$ is equivalent to $W=T_1$, and we are done. So, assume that $\deg(g)<i-1$. Set $d=i-e$, and let $W$ have a $\mathsf k$-basis $\{ f_1, \ldots, f_\delta\}$. We may further assume that the initial term of $f_j$ is $u^{r_j} v^{d-r_j}$, with $r_1>r_2>\cdots > r_\delta$, and moreover, after further reduction we may assume that for $j'\neq j$ the monomial $u^{r_{j'}}v^{d-r_{j'}}$ does not occur in $f_j$. Set $A=\{r_1, \ldots, r_\delta\}$ and $A+1=\{r+1 \mid r \in A\}$. Recall that for a $\mathsf k$-subspace $U$ of $T_m$, the set of initial terms of the nonzero elements of $U$ has precisely $\dim_{\mathsf k}(U)$ elements. 

Multiplying the elements $f_j$ by $u$ and by $v$, we see that the initial terms of the elements of $T_1 \cdot W$ contain the monomials $u^{r}v^{d+1-r}$ for $r \in A \cup (A+1)$, and hence $\dim_{\mathsf k}(T_1\cdot W) \geq \vert A \cup (A+1)\vert $. If $A$ is a union of $k$ maximal blocks of consecutive integers, then observe that $|A\cup (A+1)| = \delta+k$. 
Since $\dim_{\mathsf k}(T_1 \cdot W)=\delta+1$, we get $k=1$, i.e., $A$ consists of consecutive integers. Also, we must have $r_1=d$, since otherwise every monomial occurring in each $f_j$ would be divisible by $v$, which contradicts $\gcd(W)=1$. Thus, $A=\{p+1, p+2, \ldots, d\}$, where $p= d-\delta$. 

If $\delta=d+1$, then $W=T_d$ and we are done. So, assume that $\delta \leq d$, i.e., $p\geq 0$. We show that this leads to a contradiction. For $p+1\leq r \leq d$, let $f_r$ denote the basis element with initial term $u^rv^{d-r}$. Then by the reduction mentioned above, we have
$$f_r = u^rv^{d-r}+\sum_{j=0}^{p}a_{r,j}u^jv^{d-j}.$$
Since $\vert A \cup (A+1)\vert =\vert \{p+1, \ldots, d+1\}\vert =\delta+1$ and $\dim_{\mathsf k}(T_1 \cdot W)=\delta+1$, the initial terms of the nonzero elements of $T_1 \cdot W$ are precisely $u^rv^{d+1-r}$ for $p+1 \leq r \leq d+1$. Hence, if $w \in T_1 \cdot W$ is such that every monomial occurring in $w$ has $u$-degree at most $p+1$, then $w \in \mathsf k\cdot vf_{p+1}$. Indeed, such a nonzero $w$ has initial term $u^{p+1}v^{d-p}$, which is also the initial term of $vf_{p+1}$. Subtracting a suitable scalar multiple of $vf_{p+1}$ from $w$ gives an element of $T_1\cdot W$, all of whose monomials have $u$-degree at most $p$, and such an element must be zero.

Now, fix $r$ with $p+2\leq r \leq d$. Note that every monomial appearing in $vf_r-uf_{r-1} \in T_1 \cdot W$ has $u$-degree at most $p+1$. Hence, there exists $c_r \in \mathsf k$ such that
\begin{equation}\label{eq:in3.1}
    vf_r-uf_{r-1}=c_r\, vf_{p+1}.
\end{equation}
Finally, every monomial occurring in $f_{p+1}$ has $u$-degree at most $p+1$, and hence its $v$-degree is at least $d-p-1=\delta-1$. Thus, $v^{\delta-1}\mid f_{p+1}$, and we have $f_{p+1}=v^{\delta-1}q$ for some $q$ with $\deg(q)=p+1$. Since the monomial $u^{p+1}$ occurs in $q$, we have $v \nmid q$. We claim that $q \mid f_r$ for all $p+1 \leq r \leq d$. This is clear for $r=p+1$. If $p+2 \leq r \leq d$ and $q \mid f_{r-1}$, then by the Equation (\ref{eq:in3.1}) above, $q \mid vf_r$. Since $v \nmid q$, we get $q \mid f_r$. Therefore, $q$ divides every element of $W$, which contradicts $\gcd(W)=1$, since $\deg(q)=p+1\geq 1$. This completes the proof.
\end{proof}

\begin{corollary}\label{cor:noninc}
Let $J \subseteq T$ be a homogeneous ideal and set $c_i = (i+1)-\dim_{\mathsf k}(J_i)$. If $J_{i_0}\neq 0$ for some $i_0$, then $c_{i+1}\leq c_i$ for all $i \geq i_0$.
\end{corollary}
\begin{proof}
Since $T$ is a standard graded polynomial ring, $J_{i_0} \neq 0$ forces $J_i\neq 0$ for all $i \geq i_0$. Fix such an $i$. Then
$T_1 \cdot J_i \subseteq J_{i+1}$, and by Lemma~\ref{lem:growth}, we get
$\dim_{\mathsf k} (J_{i+1}) \ge \dim_{\mathsf k} (J_i) + 1$. Hence
\[
  c_{i+1} = (i+2) - \dim_{\mathsf k} (J_{i+1}) \le (i+2) - \dim_{\mathsf k} (J_i) - 1 = (i+1) - \dim_{\mathsf k} (J_i) = c_i . \qedhere
\]
\end{proof}

\section{The invariant \texorpdfstring{$\gamma$}{gamma}}\label{sec:constant}

In this section, $\ell$ denotes a general linear form in $R$, and the notation is as in Notation/Remark~\ref{notrem:c_i} and Notation~\ref{notation:tspeaks}.

\begin{lemma}\label{lem:dualcorank}
For every $j \geq 0$ we have $\dim_{\mathsf k}\left((0:_R \ell)_j\right) = c_{\,d-j}$.
\end{lemma}
\begin{proof}
Let $j\geq 0$ be fixed. Consider the perfect pairing $R_j \times R_{d-j} \to R_d \cong \mathsf k$ given by the multiplication in $R$. For $a \in R_j$ we have 
$$
  \ell a = 0
  \iff \ell a \cdot R_{d-j-1} = 0
  \iff a \cdot \ell R_{d-j-1} = 0 ,
$$
where the first equivalence holds because $\ell a \in R_{j+1}$ and the pairing
$R_{j+1} \times R_{d-j-1} \to \mathsf k$ is perfect. Therefore, $(0 :_R \ell)_j$ is
exactly the orthogonal complement of the subspace $\ell R_{d-j-1} \subseteq
R_{d-j}$ under the pairing $R_j \times R_{d-j} \to \mathsf k$. Hence, we have
$$
  \dim_{\mathsf k} \left((0 :_R \ell)_j\right)
  = \dim_{\mathsf k} (R_j) - \dim_{\mathsf k} (\ell R_{d-j-1})=\dim_{\mathsf k} (R_{d-j}) - \dim_{\mathsf k} (\ell R_{d-j-1}),
$$
where the last equality holds because $h(R)$ is symmetric. The proof is therefore complete, since the right-hand side is precisely the dimension of $\operatorname{coker}(R_{d-j-1} \xrightarrow[]{\times \ell } R_{d-j})$, which equals $c_{d-j}$ by definition.
\end{proof}

\begin{proposition}\label{prop:corank}
For every $i \geq 1$ we have $c_i - c_{\,d+1-i} = h_i-h_{i-1}$.
\end{proposition}

\begin{proof}
Taking Hilbert functions in the exact sequence of Notation/Remark~\ref{notrem:c_i}(a) gives
$$\dim_{\mathsf k}\left((0:_R\ell)_{i-1}\right)-h_{i-1}+h_i-c_i = 0,$$
and $\dim_{\mathsf k}\left((0:_R\ell)_{i-1}\right) = c_{\,d+1-i}$ by Lemma~\ref{lem:dualcorank}.
\end{proof}
The relation in the proposition above is implicit in the tables of \cite[Section 3]{MiglioreZanello}.
We next prove a result on which the proof of our main theorem rests.

\begin{proposition}\label{prop:constant}
Assume $\NS_R \geq 2$. Then the following hold.
\begin{enumerate}
\item[{\rm (a)}] $I_{t+1}\neq 0$, and $J_{t+1}\neq 0$ for a general linear form $\ell$.
\item[{\rm (b)}] For $t+1\leq i \leq s$ we have $c_i = c_{\,d+1-i}$.
\item[{\rm (c)}] For $i \geq t+1$ we have $c_i \geq c_{i+1}$.
\item[{\rm (d)}] $c_{t+1}=c_{t+2}=\cdots=c_s$.
\end{enumerate}
\end{proposition}

\begin{proof}
(a) If $I_{t+1}=0$, then
$$S_R = h_{t+1}=\dim_{\mathsf k}(S_{t+1})=\binom{t+3}{2}>\binom{t+2}{2}=\dim_{\mathsf k}(S_t)\geq h_t = S_R,$$
a contradiction. Hence, $I_{t+1}\neq 0$. By Notation/Remark~\ref{notrem:c_i}(d), we have $J_{t+1}=0$ if and only if $I_{t+1}\subseteq \ell S_t$. Fix a nonzero $q \in I_{t+1}$. Since $S$ is a UFD, $q$ has only finitely many linear factors up to scalars. Since $\mathsf k$ is infinite, a general $\ell \in S_1$ does not divide $q$. Hence $J_{t+1}\neq 0$.

(b) For $t+1\leq i \leq s$ we have $i-1, i \in \{t,\dots,s\}$. Therefore, $h_{i-1}=h_i=S_R$ and Proposition~\ref{prop:corank} gives $c_i = c_{\,d+1-i}$.

(c) By (a) we have $J_{t+1}\neq 0$. Thus, (c) follows from Corollary~\ref{cor:noninc}.

(d) By (c) we have $c_{t+1}\geq c_{t+2}\geq \cdots \geq c_s$. By Lemma~\ref{lem:peaks}(a), $d+1-(t+1)=d-t=s$. Therefore (b) gives $c_{t+1}=c_s$. The two facts together prove (d).
\end{proof}

Observe that part (a) uses the assumption that $\mathsf k$ is infinite, parts (c) and (d) use it through part (a), while part (b) holds for every nonzero linear form over any field. When $\NS_R \geq 3$, the statement of part (d) in terms of the ranks of the multiplication maps is \cite[Lemma 2.15]{AAIY}, where it is proved for an arbitrary linear form and is obtained from the existence of a tight annihilating scheme, and hence from \cite[Theorem 5.31]{IarrobinoKanev}. The proof above uses neither of these two results. Note that when $\NS_R = 2$,  the range $t+1 \leq i \leq s$ consists of the single degree $s = t+1$, so parts (b) and (d) carry no information in that case.

We denote the common value $c_{t+1}=\cdots=c_s$ by $\gamma$. 

\begin{corollary}\label{cor:wlp-gamma}
Assume $\NS_R \geq 2$. Then $R$ has the weak Lefschetz property if and only if $\gamma=0$.
\end{corollary}

\begin{proof}
By Lemma~\ref{lem:peaks}(b) we have $h_{j_0}=h_{j_0+1}=S_R$. Thus, $\times \ell \colon R_{j_0}\to R_{j_0+1}$ is injective if and only if it is bijective, if and only if $c_{j_0+1}=0$. Again by Lemma~\ref{lem:peaks}(b) we have $t+1 \leq j_0+1\leq s$, so $c_{j_0+1}=\gamma$ by Proposition~\ref{prop:constant}(d). The result now follows from Remark~\ref{rem:middle}.
\end{proof}
\section{Proof of the main theorem}\label{sec:divisor}

In this section, we prove the main theorem of our article (see Theorem \ref{thm:main}) using the criterion proved in Corollary~\ref{cor:wlp-gamma}.   
\begin{proposition}\label{prop:principal}
Let $R=S/I$ be an Artinian Gorenstein algebra of codimension three with $\NS_R \ge 3$, and $\ell$ be a general linear form in $R$. Let the notation be as in Notation/Remark \ref{notrem:c_i} and Notation \ref{notation:tspeaks}, with $T=\mathsf k[u,v]$.
If $\gamma \ge 1$, then there exists a nonzero element $g \in T_{\gamma}$ such that $J_i = g \cdot  T_{\,i-\gamma}$ for $t+1 \le i \le s$.
\end{proposition}

\begin{proof}
By Proposition~\ref{prop:constant} and Notation/Remark \ref{notrem:c_i}, we have
\[
  \dim_{\mathsf k} (J_i) = (i+1) - \gamma \qquad \text{for } t+1 \le i \le s .
\]
In particular, $\dim_{\mathsf k} (J_{i+1}) = \dim_{\mathsf k} (J_i) + 1$ for
$t+1 \le i \le s-1$. Note that since $\NS_R\geq 3$, we have $t+1\leq s-1$, and hence the range $t+1 \le i \le s-1$ is nonempty. 
Now, by Proposition~\ref{prop:constant}(a), we know that $J_{t+1} \ne 0$. Thus, $\dim_{\mathsf k} (J_{t+1}) = t+2-\gamma \ge 1$. This forces $\gamma \le t+1$, and hence $\dim_{\mathsf k} (J_i) \ge 1$ for all $t+1\leq i\leq s-1$.

Fix an $i$ with $t+1 \le i \le s-1$, and let $V=J_i$ with $\dim_{\mathsf k}(V)=\delta$. Then by Lemma~\ref{lem:growth} and the fact that $T_1 \cdot J_i \subseteq J_{i+1}$, we get
\[
  \delta + 1 \;\le\; \dim_{\mathsf k} (T_1 \cdot J_i)
             \;\le\; \dim_{\mathsf k} (J_{i+1}) = \delta+1. 
\]
Thus, both the inequalities above are equalities. This forces $T_1 \cdot J_i=J_{i+1}$. From $\dim_{\mathsf k}(T_1 \cdot J_i) =\delta+1$ and
Lemma~\ref{lem:growth}, we conclude that
$J_i = g_i \cdot T_{\,i-\gamma}$, where $g_i = \gcd(J_i)$ and $\deg(g_i)=i - \delta + 1 = \gamma$. Also, from $\dim_{\mathsf k}(J_{i+1})=\delta+1$, we get $J_{i+1} = T_1 \cdot J_i = g_i \cdot T_{\,i+1-\gamma}$. Observe that this also says that for $t+1 \leq i \leq s-1$, $\gcd(J_i)$ is equal up to a nonzero scalar multiple. Thus, setting $g=g_{t+1}$, we get 
$J_i = g \cdot T_{\,i-\gamma}$ for all $i$ with $t+1 \le i \le s$.
\end{proof}

\begin{proposition}\label{prop:gcd-nonconstant}
Let $\mathsf k$ be an infinite field, and $R = S/I$ be a standard graded Artinian
Gorenstein $\mathsf k$-algebra of codimension three such that $\NS_R \geq 3$. If
$\gamma \ge 1$, then $\gcd(I_s)$ has positive degree.
\end{proposition}
\begin{proof}
The idea of the proof is to show that $I_s$ is contained in a prime ideal of height $1$ and then use the fact that height $1$ primes in $S$ are principal.

Let $f= \gcd(I_s)$ and $\deg(f)= \gamma_0$.
Let $\tilde \ell$ be a general linear form in $S$. Choosing coordinates such that $\tilde \ell =z$, we may assume that the element $g \in T_{\gamma}$ in Proposition~\ref{prop:principal} above has a lift $\tilde g \in S_{\gamma}$ such that $\tilde g$ involves only the variables $x$ and $y$. 
By Proposition~\ref{prop:principal}, the image of $I_s$ in $S/\langle \tilde \ell \rangle$ is contained in $\langle g \rangle$. Therefore, $I_s \subseteq \langle \tilde g , \tilde \ell \rangle$. 
Observe that $\tilde g, \tilde\ell$ form a regular sequence in $S$, since they are supported on disjoint sets of variables. Hence $\operatorname{ht}(\langle \tilde g, \tilde \ell\rangle) =2$, and we get $\operatorname{ht}(\langle I_s\rangle) \leq 2$.

Let $\mathfrak p_1, \ldots, \mathfrak p_n$ be the minimal primes of $\langle I_s\rangle$. 
Since $\operatorname{ht}(\langle I_s\rangle) \leq 2$, none of the $\mathfrak p_j$ equals $\mathfrak m$. 
Let, if possible, each $\mathfrak p_j$ have height $2$. Since $\mathsf k$ is infinite, $(\mathfrak p_1 \cup \cdots \cup \mathfrak p_n) \cap S_1$ is a finite union of $\mathsf k$-vector subspaces of $S_1$ with dimension at most $2$. 
Therefore, we may choose the general linear form $\tilde \ell$ above so that, in addition, it does not belong to $(\mathfrak p_1 \cup \cdots \cup \mathfrak p_n) \cap S_1$. Since $\tilde \ell$ does not belong to any minimal prime of $\langle I_s\rangle$, the ideal $\langle I_s \rangle + \langle \tilde \ell \rangle$ has height $3$. But this is a contradiction to the fact that $\langle I_s,  \tilde \ell \rangle \subseteq \langle \tilde g, \tilde \ell\rangle$ and $\operatorname{ht}(\langle \tilde g, \tilde \ell\rangle) =2$. Therefore, $\langle I_s\rangle$ has a minimal prime of height $1$.

Without loss of generality, assume that $\operatorname{ht}(\mathfrak p_1)=1$. Since $S$ is a UFD, $\mathfrak p_1 = \langle p \rangle$ for some $p \in S$ with $\deg(p) \geq 1$. Since $p$ divides every element of $I_s$, by definition of $f$, we get that $p \mid f$. Hence $\gamma_0 = \deg(f) \geq \deg(p)\geq 1$.
\end{proof}

The common factor produced in the proposition above is obtained in a similar spirit in \cite[Lemma 2.11]{MiglioreZanello}. There the equality $c_i = c_{i+1}$ in two consecutive degrees forces, by Davis's theorem, a common factor of the image of $I_{i+1}$ in $S/\langle \ell\rangle$, which is then lifted to a common factor of $I_{i+1}$. The same conclusion for $I_s$ can therefore be reached from \cite[Lemma 2.11]{MiglioreZanello}. We have nevertheless kept the route through Lemma~\ref{lem:growth}. The inequality in that lemma is used in Section~\ref{sec:constant} through Corollary~\ref{cor:noninc}, and is therefore needed in any case, and its equality case then yields the required common factor directly, with no restriction on the degrees involved. 
\begin{theorem}\label{thm:main}
Let $\mathsf k$ be an infinite field, $S=\mathsf k[x,y,z]$, and $R=S/I$ be a standard graded Artinian Gorenstein $\mathsf k$-algebra of codimension three. If $\NS_R \geq 3$, then $R$ has the weak Lefschetz property.
\end{theorem}
\begin{proof}
For the sake of contradiction, suppose that $R$ does not have the WLP. Then $\gamma \geq 1$ by Corollary~\ref{cor:wlp-gamma}. Thus, $\gcd(I_s)$ has positive degree by Proposition~\ref{prop:gcd-nonconstant}. On the other hand, by Lemma~\ref{lem:peaks}(c), we have $s>d/2$. But then Remark~\ref{rem:MZgcd} forces the elements of $I_s$ to have no common factor of positive degree.  This contradiction forces $\gamma=0$, and hence $R$ has the WLP by Corollary~\ref{cor:wlp-gamma}.
\end{proof}

The following example is due to Nagel, and it shows that the result obtained above is best possible in the generality stated (see \cite[Remark 3.3]{MiglioreZanello} and \cite{HMNW}).

\begin{example}\label{ex:char2}
 Let $\operatorname{char}(\mathsf k)=2$ and $R=\mathsf k[x,y,z]/\langle x^2,y^2,z^2\rangle$. Then $R$ is a complete intersection with $h(R)=(1,3,3,1)$. We have $d=3$, $S_R=3$, $t=1$, $s=2$ and $\NS_R=2$. Given any $\ell = ax+by+cz$, the determinant of $R_1 \xrightarrow{\times \ell} R_2$ with respect to the bases $\{x,y,z\}$ and $\{xy,xz,yz\}$ equals $-2abc=0$. Hence, $\times \ell$ is not injective for any $\ell$, and $R$ does not have the WLP. 
 
 Also, observe that here we have $c_1=2$, and for a general linear form $\ell$ the map $R_1 \to R_2$ has rank two. This gives $\gamma = c_2 = 1$. In particular, $c_t = 2 > 1 = c_{t+1}$. Therefore, we see that  the constancy in Proposition~\ref{prop:constant}(d) cannot be extended to degree $t$.
\end{example}

\begin{remark}\label{rem:twopeaks}
\hfill{}
\begin{enumerate}
\item[{\rm (a)}] Note that Corollary~\ref{cor:wlp-gamma} requires only $\NS_R\geq 2$, while Proposition~\ref{prop:principal} requires $\NS_R\geq 3$. The same is therefore true of Theorem~\ref{thm:main}. Indeed, if $\NS_R=2$, then $s=t+1$. The range $t+1\leq i\leq s$ then consists of a single degree, and Lemma~\ref{lem:growth} cannot be applied. As a result, no common factor is produced. The invariant $\gamma$ is still defined and equals $c_{j_0+1}$. Therefore, Corollary~\ref{cor:wlp-gamma} reduces to Remark~\ref{rem:middle} in this case. As Example~\ref{ex:char2} above shows, the hypothesis $\NS_R\geq 3$ is indeed  necessary, and is not only a limitation of our method.

\item[{\rm (b)}] The situation in characteristic zero is different. By \cite{BMMNZ}, the problem in codimension three reduces to compressed algebras of odd socle degree, and these have exactly two peaks. Therefore, the essential remaining question is whether $\gamma=0$ when $\NS_R=2$ and $\operatorname{char}(\mathsf k)=0$.
\end{enumerate}
\end{remark}

\end{document}